\documentclass[12 pt, reqno]{amsart}
\usepackage{mathrsfs}
\usepackage{graphicx}
\usepackage{tikz}
\usepackage{amsmath}
\usepackage{amsfonts}
\usepackage{amssymb}
\usepackage{hyperref}
\usepackage{CJK}

\newtheorem{theorem}{Theorem}
\newtheorem{lemma}[theorem]{Lemma}

\newtheorem{remark}{Remark}

\newtheorem{corollary}[theorem]{Corollary}

\usepackage{amsthm,amsmath,amssymb,url,cite,color}

\numberwithin{equation}{section}

\def\lec{\textrm{lec}}
\def\inv{\textrm{inv}}
\def\aid{\textrm{aid}}
\def\ai{\textrm{ai}}
\def\id{\textrm{id}}

\def\fix{\textrm{fix}}
\def\pix{\textrm{pix}}
\def\rix{\textrm{rix}}
\def\aix{\textrm{aix}}

\def\cont{\textrm{cont}}
\def\exc{\textrm{exc}}
\def\des{\textrm{des}}
\def\maj{\textrm{maj}}
\def\imaj{\textrm{imaj}}

\def\S{\mathfrak{S}}

\def\DD{\mathcal{D}}
\def\W{\mathcal{W}}

\def\N{\mathbb N}

\def\Q{\mathbb Q}

\begin{document}

\title[On some generalized $q$-Eulerian polynomials]{ On some generalized $q$-Eulerian polynomials}

\author{Zhicong Lin}
\address[Zhicong Lin]{Department of Mathematics and Statistics, Lanzhou University, China
  \& Institut Camille Jordan, UMR 5208 du CNRS, Universit\'{e} de Lyon, Universit\'{e} Lyon 1, France}
\email{lin@math.univ-lyon1.fr}

\date{March 11, 2013}

\begin{abstract}
The  $(q,r)$-Eulerian polynomials are 
the  $(\maj-$$\exc,\fix,\exc)$ enumerative polynomials of permutations. 
Using Shareshian and Wachs' exponential generating function of these Eulerian polynomials,
 Chung and Graham proved two symmetrical $q$-Eulerian identities and asked for bijective proofs. We provide such proofs using Foata and Han's three-variable statistic $(\inv-$$\lec,\pix,\lec)$.  
We also prove a new recurrence formula for the $(q,r)$-Eulerian polynomials and study a $q$-analogue of Chung and Graham's restricted descent polynomials. 
In particular, we obtain a generalized symmetrical identity for these restricted $q$-Eulerian polynomials 
with a combinatorial proof.

\end{abstract}

\maketitle

%

\section{Introduction}

The \emph{Eulerian polynomials} $A_n(t):=\sum_{k=0}^{n}A_{n,k} t^k$ are defined by the exponential generating function 
\begin{align}\label{eq:def1}
\sum_{n\geq 0}A_n(t)\frac{z^n}{n!}=\frac{(1-t)e^{z}}{e^{zt}-te^{z}}.
\end{align}
The coefficients $A_{n,k}$ are called
\emph{Eulerian numbers}. The Eulerian numbers arise in a variety of contexts in mathematics. Let $\S_n$ denote the set of permutations of $[n]:=\{1,2,\ldots,n\}$. For
each $\pi\in\S_n$, a value $i$, $1\leq i\leq n-1$, is an \emph{excedance} (resp.~\emph{descent}) of $\pi$ if $\pi(i)>i$ (resp.~$\pi(i)>\pi(i+1)$). Denote by $\exc(\pi)$ and $\des(\pi)$ the number of excedances and descents of $\pi$, respectively. It is well-known that the Eulerian number  $A_{n,k}$ counts  permutations in $\S_n$ with $k$ descents (or $k$ excedances), that is
$$
A_n(t)=\sum_{\pi\in\S_n}t^{\des\,\pi}=\sum_{\pi\in\S_n}t^{\exc\,\pi}.
$$
 The reader is referred to~\cite{fo,pe} for some leisurely historical introductions of Eulerian polynomials and Eulerian numbers.

 Several $q$-analogs of Eulerian polynomials with combinatorial meanings have been studied in the literature (see \cite{ca,csz,st,sw2}). Recall that the \emph{major index}, $\maj(\pi)$, of a permutation $\pi\in \S_n$ is the sum of all the descents of $\pi$, i.e.,
$\maj(\pi):=\sum_{\pi(i)>\pi(i+1)}i.$
An element $i\in[n]$ is a {\em fixed point} of $\pi\in\S_n$ if $\pi(i)=i$ and we denote by $\fix(\pi)$ the number of fixed points of $\pi$. Define the  {\em $(q,r)$-Eulerian polynomials} $A_n(t,r,q)$  by the following extension of~\eqref{eq:def1}:
\begin{equation}\label{fixversion}
\sum_{n\geq0}A_n(t,r,q)\frac{z^n}{(q;q)_n}=\frac{(1-t)e(rz;q)}{e(tz;q)-te(z;q)},
\end{equation}
where 
$(q;q)_n :=\prod_{i=1}^{n}(1-q^i)$  and $e(z;q)$ is the 
$q$-exponential function 
$
\sum_{n\geq 0}\frac{z^n}{(q;q)_n}. 
$
 The following  interpretation for $A_{n}(t,r,q)$ was given by Shareshian and Wachs~\cite{sw2,sw}:
\begin{equation}\label{fix-v}
A_{n}(t,r,q):=\sum_{\pi\in \S_n}t^{\exc\,\pi}r^{\fix\,\pi}q^{(\maj-\exc)\,\pi}.
\end{equation}
These polynomials have attracted the attention of several authors (cf.\cite{fh,fh2,fh4,fh3,hlz,hy,hx,lsw,sw0,sw3}).

 Let $A_n(t,q)=A_n(t,1,q)$. Define the \emph{$q$-Eulerian numbers} $A_{n,k}(q)$ and the \emph{fixed point $q$-Eulerian numbers} $A_{n,k}^{(j)}(q)$:
\begin{equation*}
A_{n}(t,q)=\sum_{k}  A_{n,k}(q) t^k\quad\text{and}\quad A_n(t,r,q)=\sum_{j,k}A_{n,k}^{(j)}(q)r^jt^k.
\end{equation*}
By~\eqref{fix-v}, we have the following interpretations 
\begin{equation}\label{majcase}
A_{n,k}(q)=\sum_{\pi\in\S_n \atop \exc\pi=k}q^{(\maj-\exc)\pi}\quad\text{and}\quad A_{n,k}^{(j)}(q)=\sum_{\pi\in\S_n \atop {\exc\pi=k\atop\fix\pi=j}}q^{(\maj-\exc)\pi}.
\end{equation}
Recall that the $q$-binomial coefficients ${n\brack k}_q$ are defined by 
$
{n\brack k}_q:=\frac{(q;q)_n }{(q;q)_{n-k}(q;q)_k}
$
for $0\leq k\leq n$,
and ${n\brack  k}_q=0$ if $k<0$ or $k>n$.

 Answering a question of Chung et al.~\cite{cgk},   Han et al.~\cite{hlz} found and 
proved the following symmetrical $q$-Eulerian identity: 
\begin{equation} \label{eq: th16}
\sum_{k\geq 1}{a+b\brack  k}_qA_{k,a-1}(q)=\sum_{k\geq 1}{a+b\brack  k}_qA_{k,b-1}(q),
\end{equation}
where $a,b$ are integers with $a,b\geq1$. Besides a generating function proof using~\eqref{fixversion}, a bijective proof of~\eqref{eq: th16}  was also given in~\cite{hlz}. Recently, through analytical arguments,  Chung and Graham~\cite{cg} derived from~\eqref{fixversion} the following two further symmetrical $q$-Eulerian identities:

\begin{equation} \label{eq: th17}
\sum_{k\geq 1}(-1)^k{a+b\brack  k}_qq^{{a+b-k\choose 2}}A_{k,a}(q)=\sum_{k\geq 1}(-1)^k{a+b\brack  k}_qq^{{a+b-k\choose 2}}A_{k,b}(q),
\end{equation}
\begin{equation}\label{syfixversion}
\sum_{k\geq1}{a+b+j+1\brack k}_qA_{k,a}^{(j)}(q)=\sum_{k\geq1}{a+b+j+1\brack k}_qA_{k,b}^{(j)}(q),
\end{equation}
where $a,b,j$ are integers with $a,b\geq1$ and $j\geq0$, and  asked for bijective proofs. 
Our first aim is to provide such proofs using another interpretation of $A_n(t,r,q)$ introduced by Foata and Han~\cite{fh2}, which was already shown to be successful in the bijective proof of~\eqref{eq: th16} in~\cite{hlz}.

Next, for $1\leq j\leq n$, 
we shall define the {\em restricted $q$-Eulerian polynomial} $B_{n}^{(j)}(t,q)$ by the exponential generating function:
\begin{equation}\label{restrict}
\sum_{n\geq j}B_{n}^{(j)}(t,q)\frac{z^{n-1}}{(q;q)_{n-1}}=\left(\frac{A_{j-1}(t,q)(qz)^{j-1}}{(q;q)_{j-1}}\right)\frac{e(tz;q)-te(tz;q)}{e(tz;q)-te(z;q)}.
\end{equation}
and the \emph{restricted $q$-Eulerian number} $B_{n,k}^{(j)}(q)$ by $B_{n}^{(j)}(t,q)=\sum_{k}  B_{n,k}^{(j)}(q) t^k$.
 We find the following generalized symmetrical identity for the restricted $q$-Eulerian polynomials.
 \begin{theorem}\label{th:main1} Let $a,b,j$ be integers with $a,b\geq1$ and $j\geq2$. Then
\begin{align}\label{speci}
\sum_{k\geq1}{a+b+1\brack k-1}_qB_{k,a}^{(j)}(q)=\sum_{k\geq1}{a+b+1\brack k-1}_qB_{k,b}^{(j)}(q).
\end{align}
\end{theorem}

When $q=1$, the above identity was proved  by Chung and Graham~\cite{cg}, who  also asked for a bijective proof.
We shall give a bijective proof and an analytical proof of~\eqref{speci},  the latter leads to 
a new recurrence formula for $A_n(t,r,q)$.

\begin{theorem}\label{recu}
The $(q,r)$-Eulerian polynomials satisfy the following recurrence formula: 
\begin{align}\label{recurrence2}
A_{n+1}(t,r,q)=rA_n(t,r,q)+tA_n(t,q)+t\sum_{j=1}^{n-1}{n\brack j}_qq^jA_j(t,r,q)A_{n-j}(t,q)
\end{align}
for $n\geq1$ and $A_1(t,r,q)=r$.
\end{theorem}

This  paper is organized as follows. In section~\ref{sec:bi-pr}, we review some preliminaries about the three-variable statistic $(\inv,\pix,\lec)$ and give the bijective proofs of \eqref{eq: th17} and \eqref{syfixversion}. In section~\ref{sec:rec}, we first prove Theorem~\ref{recu} and then define a new statistic called ``$\rix$",  which together with descents and \emph{admissible inversions} (a statistic on permutations which appears in the context of poset topology~\cite{sw2}) gives another interpretation of $A_n(t,r,q)$. In section~\ref{sec:sym-iden}, we give two combinatorial interpretations of $B_{n,k}^{(j)}(q)$ and two proofs of Theorem~\ref{th:main1}.

\section{Bijective proofs  of \eqref{eq: th17} and \eqref{syfixversion}}%
\label{sec:bi-pr}

\subsection{Preliminaries}
A word $w=w_1w_2\ldots w_m$ on $\N$ is called a \emph{hook} if $w_1>w_2$ and either $m=2$, or $m\geq 3$ and $w_2<w_3<\ldots<w_m$.  As shown in~\cite{ge}, each permutation $\pi =\pi_1\pi_2\dots \pi_n$ admits a unique factorization, called its \emph{hook factorization}, $p\tau_{1}\tau_{2}. . . \tau_{r}$, where $p$ is an increasing word and each factor $\tau_1$, $\tau_2$, \ldots, $\tau_k$ is a hook. 
To derive the hook factorization of a permutation, one can start from the right and factor out  each hook step by step. Denote by $\inv(w)$ the numbers of \emph{inversions} of a word $w=w_1w_2\ldots w_m$, i.e., the number of pairs $(w_i,w_j)$ such that $i<j$ and $w_i>w_j$.
Then we define
\begin{align*}
\lec(\pi):= \sum_{1\leq i\leq k}\inv (\tau_i)\quad\text{and}\quad\pix(\pi)=|p|:=\text{length of the factor $p$}.
\end{align*}
 For example, the hook factorization of $\pi=1\,\,3\,\,4\,\,14\,\,12\,\,2\,\,5\,11\,\,15\,\,8\,\,6\,\,7\,\,13\,\,9\,\,10$ is
 $$1\,\,3\,\,4\,\,14\,|\,12\,\,2\,\,5\,\,11\,\,15\,|\,8\,\,6\,\,7\,|\,13\,\,9\,\,10.$$
  Hence  $p=1\,3\,4\,14$, $\tau_1=12\,2\,5\,11\,15$, $\tau_2=8\,6\,7$, 
  $\tau_3=13\,9\,10$, $\pix(\pi)=4$ and
   $$\lec(\pi)=\inv(12\,2\,5\,11\,15)\\+\inv(8\,6\,7)+\inv(13\,9\,10)=7.$$

 Let ${\mathcal A}_{0}, {\mathcal A}_{1}, . . . , {\mathcal A}_{r}$ be a series of sets on $\N$. Denote by $\inv({\mathcal A}_{0}, {\mathcal A}_{1}, . . . , {\mathcal A}_{r})$ the number of pairs $(k, l)$ such that $k\in {\mathcal A}_{i}$, $l\in {\mathcal A}_j$, $k> l$ and $i< j$. We usually write $\cont({\mathcal A})$  the set of all letters in a word ${\mathcal A}$. So we have
$
 (\inv-\lec)\pi=\inv(\cont(p), \cont(\tau_{1}),\ldots,\cont(\tau_{r}))
$
if $p\tau_{1}\tau_{2}. . . \tau_{r}$ is the hook factorization of $\pi$.

From Foata and Han~\cite[Theorem 1.4]{fh2},
we derive the  following combinatorial interpretations of the $(q,r)$-Eulerian polynomials
\begin{align}\label{pix-v0}
A_{n}(t,r,q)
=\sum_{\pi\in \S_n}t^{\lec\pi}r^{\pix\pi}q^{(\inv-\lec)\pi}.
\end{align}
Therefore
\begin{equation}\label{pix-v}
A_{n,k}(q)=\sum_{\pi\in\S_n \atop \lec\pi=k}q^{(\inv-\lec)\pi}\quad\text{and}\quad A_{n,k}^{(j)}(q)=\sum_{\pi\in\S_n \atop {\lec\pi=k\atop\pix\pi=j}}q^{(\inv-\lec)\pi}.
\end{equation}
 It is known~\cite[Proposition~1.3.17]{st0} that  the $q$-binomial coefficient
has the interpretation
\begin{equation}\label{eq:qmul}
{n\brack  k}_{q}=\sum_{({\mathcal A}, {\mathcal B})}q^{\inv({\mathcal A}, {\mathcal B})},
\end{equation}
where the sum is over all ordered partitions $({\mathcal A}, {\mathcal B})$ of $[n]$ such that $|{\mathcal A}|=k$.
We will  give bijective proofs of  \eqref{eq: th17} and \eqref{syfixversion} using the interpretations in~\eqref{pix-v} and~\eqref{eq:qmul}. 
 \begin{remark}\label{bijfh}
In~\cite{fh2}, a bijection on $\S_n$ that carries the triplet $(\fix,\exc,\maj)$ to  $(\pix,\lec,\inv)$ was constructed without being specified. This bijection consists of two steps. The first step   (see \cite[section 6]{fh2}) uses the word analogue of \emph{Kim-Zeng's decomposition}~\cite{kz} and an updated version of \emph{Gessel-Reutenauer standardization}~\cite{gr} to construct a bijection on $\S_n$ that transforms the triplet $(\fix,\exc,\maj)$ to $(\pix,\lec,\imaj)$, where $\imaj(\pi):=\maj(\pi^{-1})$ for each permutation $\pi$. The second step (see \cite[section 7]{fh2}) uses  \emph{Foata's second fundamental transformation}~\cite{fo1}  to carry the triplet  $(\pix,\lec,\imaj)$ to $(\pix,\lec,\inv)$.  In view of this bijection, one can construct bijective proofs of  \eqref{eq: th16}, \eqref{eq: th17} and \eqref{syfixversion} using the original interpretations in~\eqref{majcase},  through the bijective proof of~\eqref{eq: th16} in~\cite{hlz}  and our bijective proofs.    
\end{remark}

To construct our bijective proofs, 
we need two elementary transformations from~\cite{hlz} that we recall now. 
Let $\tau$ be a hook  with $\inv(\tau)=k$ and  $\cont(\tau)=\{x_1, \ldots, x_m\}$, where  $x_1<\ldots <x_m$. 
Define
\begin{equation}\label{def:d}
d(\tau)=x_{m-k+1}x_1\ldots x_{m-k} x_{m-k+2}\ldots x_m.
\end{equation}
Clearly, $d(\tau)$  is the unique hook with  $\cont(d(\tau))=\cont(\tau)$
 and satisfying  
$$
 \inv(d(\tau))=m-k=|\cont(\tau)|-\inv(\tau).
$$
  Let $\tau$ be a hook or an increasing word with $\inv(\tau)=k$ and  $\cont(\tau)=\{x_1, \ldots, x_m\}$, where  $x_1<\ldots <x_m$. 
Define
\begin{equation}\label{def:dd}
d'(\tau)=x_{m-k}x_1\ldots x_{m-k-1} x_{m-k+1}\ldots x_m.
\end{equation}
It is not difficult to see that, $d'(\tau)$  is the unique hook (when $k<m-1$) or increasing word (when $k=m-1$) with $\cont(d(\tau))=\cont(\tau)$ and 
 satisfying  
 $$
 \inv(d(\tau))=m-k-1=|\cont(\tau)|-1-\inv(\tau).
$$
 
 \subsection{ Bijective proof of \eqref{eq: th17}}

Let $\S_n(k)=\{\pi\in\S_n : \pix(\pi)=k\}$ and $\DD_n=\S_n(0)$.  We first notice that the left-hand side of~\eqref{eq: th17} has the following interpretation:
 \begin{equation}\label{pixve}
 \sum_{\pi\in\DD_{n}\atop\lec\pi=a}q^{(\inv-\lec)\pi}=\sum_{k\geq 1}(-1)^{n-k}{n\brack  k}_qq^{{n-k\choose 2}}A_{k,a}(q).
 \end{equation}
 This interpretation follows immediately from~\cite[Corollary 4.4]{sw} and~\eqref{pix-v0}.
One can also give a direct combinatorial proof similarly as in~\cite{wm}. Actually, by \eqref{pix-v} and~\eqref{eq:qmul} we have 
 \begin{align*}\label{open}
 A_{n,a}(q)&=\sum_{\pi\in\S_n \atop \lec\pi=a}q^{(\inv-\lec)\pi}\nonumber\\
 &=\sum_k\sum_{\pi\in\S_n(k) \atop \lec\pi=a}q^{\inv(\cont(p),\cont(\tau_1\ldots\tau_r))+\inv(\cont(\tau_{1}),\cont(\tau_2),\ldots,\cont(\tau_{r}))}\nonumber\\
 &=\sum_{k}\sum_{{\mathcal A}\subseteq[n]\atop|{\mathcal A}|=k}q^{\inv({\mathcal A},[n]\setminus{\mathcal A})}\sum_{\pi\in\DD_{n-k}\atop\lec\pi=a}q^{(\inv-\lec)\pi}\nonumber\\
 &=\sum_k{n\brack k}_q\sum_{\pi\in\DD_{k}\atop\lec\pi=a}q^{(\inv-\lec)\pi}.
 \end{align*}
 Applying  Gaussian inversion (or $q$-binomial inversion) to the above identity we obtain~\eqref{pixve}.

Now, by~\eqref{pixve}, the symmetrical identity \eqref{eq: th17} is equivalent to   
the $j=0$ case of the following Lemma. 
  \begin{lemma}\label{sypix}
 For $0\leq j\leq n$, there is an involution ${\bf v} \mapsto {\bf u}$ on $\S_n(j)$ satisfying 
\begin{align*}
\lec({\bf u})=n-j-\lec({\bf v})\quad \textrm{and}\quad
(\inv-\lec){\bf u}  = (\inv-\lec){\bf v}.
\end{align*}
 \end{lemma}
 \begin{proof}
 Let ${\bf v}=p\tau_{1}\tau_{2}\ldots \tau_{r}$ be the hook factorization of ${\bf v}\in\S_n(j)$, where $p$ is an increasing word and each factor $\tau_1$, $\tau_2$, \ldots, $\tau_r$ is a hook. We define ${\bf u}=pd(\tau_{1})\ldots d(\tau_{r})$, where $d$ is defined in~\eqref{def:d}. It is easy to check that this mapping is an involution on $\S_n(j)$ with the desired properties.  
 \end{proof}
 
 By~\eqref{pix-v}, Lemma~\ref{sypix} gives a simple bijective proof of the following known~\cite{cg,sw} symmetric property of the fixed point $q$-Eulerian numbers.
 \begin{corollary}\label{th:2} For $n,k,j\geq0$,
\begin{equation}\label{syC}
A_{n,k}^{(j)}(q)=A_{n,n-j-k}^{(j)}(q).
\end{equation} 
\end{corollary}

\subsection{ Bijective proof of \eqref{syfixversion}}
Recall~\cite{hlz} that, for a fixed positive integer $n$, a {\it two-pix-permutation of $[n]$} is
a  sequence of words
 \begin{equation}
\label{eq:defvec}
{\bf v}=(p_1, \tau_1, \tau_2, \ldots, \tau_{r-1}, \tau_{r}, p_2)
\end{equation}
satisfying  the following conditions:
\begin{itemize}
\item[(C1)] $p_1$ and $p_2$ are two increasing words, possibly empty;
\item[(C2)] $\tau_1, \ldots, \tau_r$ are hooks for some positive integer $r$;
\item[(C3)] The concatenation $p_1 \tau_1 \tau_2 \ldots \tau_{r-1} \tau_{r} p_2$ of all components of $\bf v$ is a permutation of $[n]$.
\end{itemize}
We also extend the two statistics to the two-pix-permutations by
\begin{equation*}
\lec ({\bf v}) =\sum_{1\leq i\leq r} \inv(\tau_i)\quad\text{and}\quad
\inv ({\bf v}) =\inv(p_1 \tau_1 \tau_2 \ldots \tau_{r-1} \tau_{r} p_2).
\end{equation*}
It follows  that
\begin{equation}\label{eq:il}
(\inv-\lec) {\bf v} = 
\inv(\cont(p_1), \cont(\tau_1),\cont(\tau_2),\ldots, \cont(\tau_r), \cont(p_2) ).
\end{equation}
Let $\W_n(j)$ denote the set of all two-pix-permutations with $|p_1|=j$.

\begin{lemma} \label{th:3}
Let $a,j$ be fixed nonnegative integers. Then
\begin{equation} \label{eq:comb}
\sum_{{\bf v}\in\W_n(j)\atop\lec{\bf v}=a}q^{(\inv-\lec){\bf v}}=\sum_{k\geq 1}{n\brack k}_qA_{k,a}^{(j)}(q).
\end{equation}
\end{lemma}
\begin{proof}
By the hook factorization, the  two-pix-permutation 
in~\eqref{eq:defvec}
is in bijection with the pair 
$(\sigma, p_2)$, 
where
$\sigma= p_1 \tau_1 \tau_2 \ldots \tau_{r-1} \tau_{r}$ is a permutation on $[n]\setminus\cont(p_2)$ and $p_2$  is an increasing word.  Thus, by \eqref{pix-v}, \eqref{eq:qmul} and \eqref{eq:il}, 
the generating function of all two-pix-permutations ${\bf v}$ 
of $[n]$ with $|p_1|=j$ such that $\lec({\bf v})=a$ and $|p_2|=n-k$ with respect to the weight
$q^{(\inv-\lec){\bf v}}$ is
${n\brack n-k}_qA_{k,a}^{(j)}(q)$.
\end{proof}

\begin{lemma}\label{th: 4}
Let $j$ be a fixed nonnegative integer. Then
there is an involution ${\bf v} \mapsto {\bf u}$ 
on $\W_n(j)$ satisfying
\begin{align*}
\lec ({\bf v}) = n -j- 1 - \lec ({\bf u}),\quad \textrm{and}\quad
(\inv-\lec) {\bf v} = (\inv-\lec) {\bf u}.
\end{align*}
\end{lemma}

\begin{proof}
We give an explicit construction of  the bijection using the involutions $d$ and $d'$ defined in~\eqref{def:d} and~\eqref{def:dd}. 

Let
${\bf v}=(p_1, \tau_1, \tau_2, \ldots, \tau_{r-1}, \tau_{r}, p_2)$ be a two-pix-permutation of $[n]$ with $|p_1|=j$.
If $p_2\neq\emptyset$, then 
$$
{\bf u}=(p_1, d(\tau_1), d(\tau_2), \ldots, d(\tau_{r-1}), d(\tau_{r}), d'(p_2)),
$$
otherwise, 
$$
{\bf u}=(p_1, d(\tau_1), d(\tau_2), \ldots, d(\tau_{r-1}), d'(\tau_{r})).
$$
As $d$ and $d'$ are involutions, this mapping is an involution on $\W_n(j)$.

Since we have $\lec (d(\tau_i)) = |\cont(\tau_i)| - \lec(\tau_i)$ for $1\leq i\leq r$ and
$\lec (d'(p_2)) = |\cont(p_2)|-1$ in the  case $p_2\neq\emptyset$,
it follows that
$\lec ({\bf u}) = \sum_{i=1}^r  |\cont(\tau_i)| + |\cont(p_2)|-1 - \lec ({\bf v})
=n-j-1 - \lec ({\bf v}).
$
The above identity is also valid when $p_2=\emptyset$.

Finally it follows from \eqref{eq:il}  that $(\inv-\lec){\bf u}=(\inv-\lec){\bf v}$. This finishes the proof of the lemma.
\end{proof}

Combining Lemmas~\ref{th:3} and \ref{th: 4} we obtain a bijective proof of \eqref{syfixversion}.

\section{A  new recurrence formula for the $(q,r)$-Eulerian polynomials}
\label{sec:rec}

The \emph{Eulerian differential operator} $\delta_x$ used below is defined by
$$
\delta_x(f(x)):=\frac{f(x)-f(qx)}{x},
$$
for any $f(x)\in\Q[q][[x]]$ in the ring of formal power series in $x$ over $\Q[q]$ (instead of the traditional $(f(x)-f(qx))/((1-q)x)$, see \cite{an2,ch}). 
We need the  following elementary properties of $\delta_x$. 
\begin{lemma}\label{delta}
For any $f(x),g(x)\in\Q[q][[x]]$,
\begin{equation*}
\delta_x(f(x)g(x))=f(qx)\delta(g(x))+\delta(f(x))g(x)
\end{equation*}
and
\begin{equation*}
\delta_x\left(\frac{1}{f(x)}\right)=\frac{-\delta_x(f(x))}{f(qx)f(x)}\quad(\text{$f(x)\neq0$}).
\end{equation*}
\end{lemma}

\begin{proof}[{\bf Proof of Theorem~\ref{recu}}]
 It is not difficult to show that, for any variable $a$
\begin{equation*} 
\delta_z(e(az;q))=ae(az;q).
\end{equation*}
Now, applying $\delta_z$ to both sides of~\eqref{fixversion} and using the above property and Lemma~\ref{delta}, we obtain
\begin{align*}
&\sum_{n\geq0}A_{n+1}(t,r,q)\frac{z^n}{(q;q)_n}\\
=&\delta_z\left(\frac{(1-t)e(rz; q)}{e(tz; q)-te(z;q)}\right)\\
=&\delta_z((1-t)e(rz; q))(e(tz; q)-te(z;q))^{-1}+\delta_z\left((e(tz; q)-te(z;q))^{-1}\right)(1-t)e(rzq; q)\\
=&\frac{r(1-t)e(rz; q)}{e(tz; q)-te(z;q)}+\frac{(1-t)e(rzq; q)(te(z; q)-te(tz; q))}{(e(tqz; q)-te(qz;q))(e(tz; q)-te(z;q))}\\
=&r\sum_{n\geq0}A_{n}(t,r,q)\frac{z^n}{(q;q)_n}+t\left(\sum_{n\geq0}A_{n}(t,r,q)\frac{(qz)^n}{(q;q)_n}\right)\left(\sum_{n\geq1}A_{n}(t,q)\frac{z^n}{(q;q)_n}\right).
\end{align*}
Taking the coefficient of $\frac{z^n}{(q;q)_n}$ in both sides of  the above equality, we get~\eqref{recurrence2}. 
\end{proof}

\begin{remark}
A different recurrence formula for $A_{n}(t,r,q)$ was obtained in \cite[Corollary 4.3]{sw}. Eq.~\eqref{recurrence2} is similar to two recurrence formulas in the literature: one for the $(\inv,\des)$-$q$-Eulerian polynomials in~\cite[Corollary 2.22]{pa} (see also~\cite{ch}) and the other one for the $(\maj,\des)$-$q$-Eulerian polynomials in~\cite[Corollary 3.6]{pa}.
 \end{remark}
 
 We shall give another interpretation of $A_n(t,r,q)$ in the following.
  
  Let $\pi\in\S_n$. Recall that an \emph{inversion} of $\pi$ is a pair $(\pi(i),\pi(j))$ such that $1\leq i<j\leq n$ and $\pi(i)>\pi(j)$. An \emph{admissible inversion} of $\pi$ is an inversion $(\pi(i),\pi(j))$ that satisfies either
 \begin{itemize}
 \item $1<i$ and $\pi(i-1)<\pi(i)$ or 
 \item there is some $l$ such that $i<l<j$ and $\pi(i)<\pi(l)$.
 \end{itemize}
  We write $\ai(\pi)$ the number of admissible inversions of $\pi$. Define the statistic
  $$
  \aid(\pi):=\ai(\pi)+\des(\pi).
  $$
 For example, if $\pi=42153$ then there are $5$ inversions, but only $(4,3)$ and $(5,3)$ are admissible. So $\inv(\pi)=5$, $\ai(\pi)=2$ and $\aid(\pi)=2+3=5$.
 The statistics $\ai$ and $\aid$ were first studied by Shareshian and Wachs~\cite{sw2} in the context of Poset Topology. Here we follow the definitions in~\cite{lsw}. The curious result that the pairs $(\aid,\des)$ and $(\maj,\exc)$ are equidistributed on $\S_n$ was proved in~\cite{lsw} using techniques of Rees products and lexicographic shellability.

  Let  ${\mathcal W}$ be the set of all the words on $\N$. We define a new statistic, denoted by ``$\rix$", on ${\mathcal W}$ recursively.  Let $W=w_1w_2\cdots w_n$ be a word in ${\mathcal W}$ and $w_i$ be the rightmost maximum element of $W$. We define $\rix(W)$ by (with convention that $\rix(\emptyset)=0$)
$$ \rix(W):=
 \begin{cases}
0,&\text{if $i=1\neq n$,}\\
1+\rix(w_1\cdots w_{n-1}),&\text{if $i=n$},\\
\rix(w_{i+1}w_{i+2}\cdots w_n),& \text{if $1<i<n$.}
\end{cases}
$$
 For example, we have $\rix(1\,5\,2\,4\,3\,3\,5)=1+\rix(1\,5\,2\,4\,3\,3)=1+\rix(2\,4\,3\,3)=1+\rix(3\,3)=2+\rix(3)=3.$
 As every permutation can be viewed as a word on $\N$, this statistic is well-defined on permutations.

We write 
$\S_n^{(j)}$ the set of permutations $\pi\in\S_n$ with $\pi(j)=n$. 
For $n\geq1$ and $1\leq j\leq n$, we define
$
B_{n}(t,r,q):=\sum_{\pi\in \S_n}t^{\des\,\pi}r^{\rix\,\pi}q^{\ai\,\pi}
$
and its restricted version by
\begin{equation}\label{restrict2}
B_{n}^{(j)}(t,r,q):=\sum_{\pi\in \S_n^{(j)}}t^{\des\,\pi}r^{\rix\,\pi}q^{\ai\,\pi}.
\end{equation}
We should note here that the restricted $q$-Eulerian polynomial $B_n^{(j)}(t,q)$ is  some modification of $B_{n}^{(j)}(t,1,q)$, as will be shown in the next section.

\begin{theorem} \label{aiddes} We have the following interpretation for $(q,r)$-Eulerian polynomials:
\begin{equation}\label{B}
A_n(t,r,q)=\sum_{\pi\in \S_n}t^{\des\,\pi}r^{\rix\,\pi}q^{\ai\,\pi}.
\end{equation}
\end{theorem} 
 \begin{proof} 
 We will show that $B_{n}(t,r,q)$ satisfies the same recurrence formula and initial condition as $A_{n}(t,r,q)$.
 For $n\geq1$, it is clear from the definition of $B_n(t,r,q)$ that 
 \begin{align}\label{sumB}
B_{n+1}(t,r,q)=\sum_{1\leq j\leq n+1}B_{n+1}^{(j)}(t,r,q).
 \end{align}
It is easy to see that
\begin{equation}\label{1}
B_{n+1}^{(1)}(t,r,q)=tB_n(t,1,q)\quad\text{and}\quad B_{n+1}^{(n+1)}(t,r,q)=rB_n(t,r,q).
\end{equation}
We then consider $B_{n+1}^{(j)}(t,r,q)$ for the case of $1<j<n+1$.

For a set $X$, we denote by ${X \choose m}$ the $m$-element subsets of $X$ and $\S_X$ the set of permutations of $X$.
Let $\W(n,j)$ be the set of all triples $(W,\pi_1,\pi_2)$ such that $W\in{[n] \choose j}$ and $\pi_1\in\S_W,\pi_2\in\S_{[n]\setminus W}$.   
It is not difficult to see that the mapping $\pi\mapsto (W,\pi_1,\pi_2)$ defined by
\begin{itemize}
\item $W=\{\pi(i) : 1\leq i\leq j-1\}$,
\item $\pi_1=\pi(1)\pi(2)\cdots\pi(j-1)$ and $\pi_2=\pi(j+1)\pi(j+2)\cdots\pi(n)$
\end{itemize}
is a bijection between $\S_n^{(j)}$ and $\W(n-1,j-1)$ and satisfies 
\begin{equation*}
\des(\pi)=\des(\pi_1)+\des(\pi_2)+1,\quad \rix(\pi)=\rix(\pi_2)
\end{equation*}
and 
$$
\ai(\pi)=\ai(\pi_1)+\ai(\pi_2)+\inv(W,[n-1]\setminus W)+n-j.
$$
Thus, for $1<j<n+1$, we have
\begin{align}\label{resB}
B_{n+1}^{(j)}(t,r,q)
&=\sum_{\pi\in\S_{n+1}^{(j)}}t^{\des\pi}r^{\rix\pi}q^{\ai\pi}\nonumber\\
&=tq^{n+1-j}\sum_{(W,\pi_1,\pi_2)\in\W(n,j-1)}q^{\inv(W,[n]\setminus W)}q^{\ai(\pi_1)}t^{\des(\pi_1)}r^{\rix(\pi_2)}q^{\ai(\pi_2)}t^{\des(\pi_2)}\nonumber\\
&=tq^{n+1-j}\sum_{W\in{[n]\choose j-1}}q^{\inv(W,[n]\setminus W)}\sum_{\pi\in\S_W}q^{\ai(\pi_1)}t^{\des(\pi_1)}\sum_{\pi_2\in\S_{[n]\setminus W}}r^{\rix(\pi_2)}q^{\ai(\pi_2)}t^{\des(\pi_2)}\nonumber\\
&=tq^{n+1-j}{n\brack j-1}_qB_{j-1}(t,1,q)B_{n+1-j}(t,r,q),
\end{align}
where we apply \eqref{eq:qmul} to the last equality. Substituting~\eqref{1} and~\eqref{resB} into~\eqref{sumB} we obtain
$$
B_{n+1}(t,r,q)=rB_n(t,r,q)+tB_n(t,1,q)+t\sum_{j=1}^{n-1}{n\brack j}_qq^jB_j(t,r,q)B_{n-j}(t,1,q).
$$
By Theorem~\ref{recu}, $B_n(t,r,q)$ and $A_n(t,r,q)$ satisfy the same recurrence formula and initial condition, thus $B_n(t,r,q)=A_n(t,r,q)$.  This finishes the proof of the theorem.
\end{proof}

It follows from~\eqref{fix-v},~\eqref{pix-v0} and~\eqref{B} that 

\begin{corollary}
The three triplets $(\rix,\des,\aid)$, $(\fix,\exc,\maj)$ and $(\pix,\lec,\inv)$ are equidistributed on $\S_n$.
\end{corollary}
\begin{remark}
 At the \emph{Permutation Patterns 2012} conference, Alexander Burstein~\cite{ab} gave a direct bijection on $\S_n$ that transforms the triple $(\rix,\des,\aid)$ to $(\pix,\lec,\inv)$. The new statistic ``$\rix$" was introduced independently therein under the name ``\aix". Actually, the definitions of both are slightly different, but they are the same up to an easy transformation.
  It would be very interesting to find a similar bijective proof of the equidistribution of  $(\rix,\des,\aid)$ and $(\fix,\exc,\maj)$. See also Remark~\ref{bijfh}.
 \end{remark}

\section{A symmetrical identity for restricted $q$-Eulerian polynomials}
\label{sec:sym-iden}

\subsection{An interpretation of $B_{n,k}^{(j)}(q)$ and a proof of Theorem~\ref{th:main1}}

It follows from~\eqref{fixversion} and~\eqref{restrict}  that $B_{1,0}^{(1)}(q)=1$ and $B_{n,k}^{(1)}(q)=A_{n-1,k-1}(q)$ for $k\geq1$. For $j\geq2$, we have the following interpretation for $B_{n,k}^{(j)}(q)$.

\begin{lemma} \label{le:4} 
For $2\leq j\leq n$,
$B_{n,k}^{(j)}(q)=\sum_{\pi\in\S_n^{(j)}\atop\des(\pi)=k}q^{\ai(\pi)+2j-n-1}.$
\end{lemma}
\begin{proof}  When $j\geq2$, by the recurrence relation~\eqref{resB}, one can compute without difficulty that the  exponential generating function  $\sum_{n\geq j}q^{2j-n-1}B_n^{(j)}(t,1,q)\frac{z^{n-1}}{(q;q)_{n-1}}$ is exactly the right side of~\eqref{restrict} using~\eqref{fixversion} and~\eqref{B}, which would finish the proof of the lemma.
\end{proof}

Originally,  the \emph{restricted Eulerian number} $B_{n,k}^{(j)}$ in~\cite{cg}  was defined  to be the number of permutations $\pi\in\S_n$ with $\des(\pi)=k$ and $\pi(j)=n$. According to the above lemma, $B_{n,k}^{(j)}(q)$ is really a $q$-analogue of $B_{n,k}^{(j)}$. This justifies the names restricted $q$-Eulerian number and restricted $q$-Eulerian polynomials.

\begin{lemma} \label{reflection}For $1<j<n$, we have 
\begin{equation*}
B_{n,k}^{(j)}(q)=B_{n,n-1-k}^{(j)}(q).
\end{equation*}
\end{lemma}
\begin{proof}
We first construct an involution $f: \pi \mapsto \pi'$ on $\S_n$  satisfying
\begin{equation}\label{prodes}
\ai(\pi)=\ai(\pi')\quad\text{and}\quad\des(\pi)=n-1-\des(\pi').
\end{equation}
For $n=1$, define $f(\id)=\id$. For $n\geq2$, suppose that $\pi=\pi_1\cdots\pi_n$ is a permutation of $\{\pi_1,\cdots,\pi_n\}$ and $\pi_j$ is the maximum element in $\{\pi_1,\cdots,\pi_n\}$. We construct $f$ recursively as follows
$$f(\pi)=
\begin{cases}
f(\pi_2\pi_3\cdots\pi_n)\,\pi_1,&\text{if $j=1$,}\\
\pi_n\,f(\pi_1\pi_2\cdots\pi_{n-1}),&\text{if $j=n$},\\
f(\pi_1\pi_2\cdots\pi_{j-1})\,\pi_j\,f(\pi_{j+1}\pi_{j+2}\cdots\pi_{n}),& \text{otherwise.}
\end{cases}
$$
For example, if $\pi=3\,2\,5\,7\,6\,4\,1$, then
$$
f(\pi)=f(3\,2\,5)\,7\,f(6\,4\,1)=5\,f(3\,2)7\,f(4\,1)\,6=5\,2\,3\,7\,1\,4\,6.
$$
Clearly, 
$\ai(\pi)=7=\ai(\pi')$ and $\des(\pi)=4=7-1-\des(\pi')$.
It is not difficult to see that $f$ is an involution. We can show that $f$ satisfies~\eqref{prodes} by induction on $n$, which is routine and left to the reader.

 For each $\pi=\pi_1\cdots\pi_{j-1}\,n\,\pi_{j+1}\cdots\pi_n$ in $\S_n^{(j)}$, we then define 
$$
g(\pi)=f(\pi_1\cdots\pi_{j-1})\,n\,f(\pi_{j+1}\cdots\pi_n).
$$
As $f$ is an involution, $g$ is an involution on $\S_n^{(j)}$. 
It follows from~\eqref{prodes} that 
$
\ai(g(\pi))=\ai(\pi)
$
and $\des(\pi)=n-1-\des(g(\pi))$,
which completes the proof in view of Lemma~\ref{le:4}.
\end{proof}
\begin{remark}
A bijective proof of Lemma~\ref{reflection} when $q=1$ was given in~\cite{cg}. But their bijection does not preserve the admissible inversions. Supposing that $\pi=\pi_1\cdots\pi_n$ is a permutation of $\{\pi_1,\cdots,\pi_n\}$ and $\pi_j$ is the maximum element in $\{\pi_1,\cdots,\pi_n\}$, we modify $f$ defined above to $f'$ as follows:
$$f'(\pi)=
\begin{cases}
f'(\pi_2\pi_3\cdots\pi_n)\,\pi_1,&\text{if $j=1$,}\\
\pi,&\text{if $j=n$},\\
f'(\pi_1\pi_2\cdots\pi_{j-1})\,\pi_j\,f'(\pi_{j+1}\pi_{j+2}\cdots\pi_{n}),& \text{otherwise.}
\end{cases}
$$
 The reader is invited to check that $f'$ would provide another bijective proof of Corollary~\ref{th:2} using  $(\des,\rix,\ai)$. 
\end{remark}

Now we are in position to give a generating function proof of Theorem~\ref{th:main1}.

\begin{proof}[{\bf Proof of Theorem~\ref{th:main1}}]
We start with the generating function given in~\eqref{restrict}. Multiplying both sides by $e(tz;q)-te(z;q)$, we obtain
$$
\sum_{n,k}B_{n,k}^{(j)}(q)t^k\frac{z^{n-1}}{(q;q)_{n-1}}(e(tz;q)-te(z;q))=\frac{(qz)^{j-1}A_{j-1}(t,q)}{(q;q)_{j-1}}(e(tz;q)-te(tz;q)).
$$
Expanding the exponential functions, we have 
\begin{align*}
\sum_{n,k,i}B_{n,k}^{(j)}(q)\frac{t^{k+i}z^{n+i-1}}{(q;q)_i(q;q)_{n-1}}&-\sum_{n,k,i}B_{n,k}^{(j)}(q)\frac{t^{k+1}z^{n+i-1}}{(q;q)_i(q;q)_{n-1}}\\
&=\frac{(qz)^{j-1}A_{j-1}(t,q)}{(q;q)_{j-1}}\sum_{n\geq0}\frac{(1-t)t^nz^n}{(q;q)_n}.
\end{align*}
Identifying the coefficient of $t^lz^{m-1}$ gives
\begin{align*}
\sum_{k}\frac{B_{m+k-l,k}^{(j)}(q)}{(q;q)_{l-k}(q;q)_{m+k-l-1}}&-\sum_{i}\frac{B_{m-i,l-1}^{(j)}(q)}{(q;q)_i(q;q)_{m-i-1}}\\
&=\frac{q^{j-1}\left(A_{j-1,l+j-m}(q)-A_{j-1,l+j-m-1}(q)\right)}{(q;q)_{j-1}(q;q)_{m-j}}.
\end{align*}
Multiplying both sides by $(q;q)_{m-1}$, we get
\begin{align*}
\sum_kB_{m+k-l,k}^{(j)}(q){m-1\brack l-k}_q&-\sum_iB_{m-i,l-1}^{(j)}(q){m-1\brack i}_q \\
=&(A_{j-1,l+j-m}(q)-A_{j-1,l+j-m-1}(q))q^{j-1}{m-1\brack j-1}_q.
\end{align*}
Changing the variables of the two summations on the left side gives
 \begin{align}\label{16}
\sum_kB_{k,k+l-m}^{(j)}(q){m-1\brack k-1}_q&-\sum_kB_{k,l-1}^{(j)}(q){m-1\brack k-1}_q \nonumber\\
&=(A_{j-1,l+j-m}(q)-A_{j-1,l+j-m-1}(q))q^{j-1}{m-1\brack j-1}_q.
\end{align}
We apply the symmetric property in Lemma~\ref{reflection} to the first summation on the left side of~\eqref{16} and we have
\begin{align}\label{17}
&\sum_kB_{k,k+l-m}^{(j)}(q){m-1\brack k-1}_q\nonumber\\
=&B_{j,j+l-m}^{(j)}(q){m-1\brack j-1}_q
+\sum_{k\neq j}B_{k,m-1-l}^{(j)}(q){m-1\brack k-1}_q.
\end{align}
 It follows from Lemma~\ref{le:4} and Theorem~\ref{aiddes} that
 $$
 B_{n,k}^{(n)}(q)=\sum_{\pi\in\S_n^{(n)}\atop\des(\pi)=k}q^{\ai(\pi)+n-1}=q^{n-1}A_{n-1,k}(q).
 $$ 
 Using the symmetric property of $A_{n,k}(q)$, that is $A_{n,k}(q)=A_{n,n-1-k}(q)$, and the above property, the right side of~\eqref{16} can be treated as follows:
\begin{align}\label{18}
&(A_{j-1,l+j-m}(q)-A_{j-1,l+j-m-1}(q))q^{j-1}{m-1\brack j-1}_q\nonumber\\
=&B_{j,j+l-m}^{(j)}(q){m-1\brack j-1}_q-A_{j-1,m-1-l}(q)q^{j-1}{m-1\brack j-1}_q\nonumber\\
=&B_{j,j+l-m}^{(j)}(q){m-1\brack j-1}_q-B_{j,m-1-l}^{(j)}(q){m-1\brack j-1}_q.
\end{align}
Now we substitute  \eqref{17}, \eqref{18} into \eqref{16} and obtain
\begin{equation*}
\sum_{k}B_{k,m-1-l}^{(j)}(q){m-1\brack k-1}_q=\sum_kB_{k,l-1}^{(j)}(q){m-1\brack k-1}_q,
\end{equation*}
which becomes~\eqref{speci} after setting $m=a+b+2$ and $l-1=b$.
\end{proof}
\begin{remark} \label{remark1} The only case that is left out in Theorem~\ref{th:main1} is the case of $j = 1$. However, as $B_{n,k}^{(1)}(q)=A_{n-1,k-1}(q)$, the corresponding symmetrical identity for this case is~\eqref{eq: th16}.
\end{remark}

\subsection{Another interpretation of $B_{n,k}^{(j)}(q)$ and a bijective proof of Theorem~\ref{th:main1}}
 Let 
 $$\bar{\S}_n^{(j)}:=\{\pi\in\S_n: \pi(j+1)=1\}\,\,\text{for $1\leq j< n$}\quad\text{and}\quad
 \bar{\S}_n^{(n)}:=\{\pi' \square 1: \pi'\in\S_{[n]\setminus\{1\}}\}.$$
  The $``\square"$ in  $\pi=\pi_1\pi_2\cdots\pi_{n-1}\square1\in\bar{\S}_n^{(n)}$ means that the $n$-th position of $\pi$ is empty and the hook factorization of $\pi$ is defined to be $p\tau_1\cdots\tau_r\square1$, where $p\tau_1\cdots\tau_r$ is the hook factorization of $\pi_1\cdots\pi_{n-1}$ and $``\square1"$ is  viewed as a hook.  We also define the statistics
  $$\lec(\pi_1\pi_2\cdots\pi_{n-1}\square1)=\sum_{i=1}^r\lec(\tau_i)\quad\text{and}\quad\inv(\pi_1\pi_2\cdots\pi_{n-1}\square1)=\inv(\pi_1\pi_2\cdots\pi_{n-1}1).$$ For example, $\bar{\S}_3^{(3)}=\{32\square1, 23\square1\}$ with $\lec(32\square1)=1, \lec(23\square1)=0$, $\inv(32\square1)=3$, and $\inv(23\square1)=2$. 
\begin{lemma}\label{le:6}
For $1\leq j\leq n$,
$B_{n,k}^{(j)}(q)=
\sum_{\pi\in\bar{\S}_n^{(j)}\atop\lec(\pi)=k}q^{(\inv-\lec)\pi}.
$
\end{lemma}
\begin{proof}
Let $\bar{B}_n^{(j)}(t,q):=\sum_{\pi\in\bar{\S}_n^{(j)}}q^{(\inv-\lec)\pi}t^{\lec\,\pi}$. We recall that,  to derive the hook factorization of a permutation, one can start from the right and factor out  each hook step by step. Therefore, the hook factorization of $\pi=\pi_1\cdots\pi_{j-1}\pi_j1\pi_{j+2}\cdots\pi_n$ in $\pi\in\bar{\S}_n^{(j)}$ is $p\tau_1\cdots\tau_s\tau'_1\cdots\tau'_r$, where $p\tau_1\cdots\tau_s$ and $\tau'_1\cdots\tau'_r$ are hook factorizations of $\pi_1\cdots\pi_{j-1}$ and $\pi_j1\pi_{j+2}\cdots\pi_n$,  respectively.  When $n>j$, it is not difficult to see that
$$\lec(\pi_j1\pi_{j+2}\cdots\pi_n)=1+\lec(\pi_j\pi_{j+2}\cdots\pi_n)
$$
and
$$
(\inv-\lec)(\pi_j1\pi_{j+2}\cdots\pi_n)=(\inv-\lec)(\pi_j\pi_{j+2}\cdots\pi_n).
$$ Thus by~\eqref{eq:qmul}, we have 
\begin{equation}
\bar{B}_n^{(j)}(t,q)=A_{j-1}(t,q)q^{j-1}{n-1\brack j-1}_qtA_{n-j}(t,q)
\end{equation}
for $n>j$. Clearly, $\bar{B}_j^{(j)}(t,q)=A_{j-1}(t,q)q^{j-1}$. So, by~\eqref{fixversion}, the exponential generating function $\sum_{n\geq j}\bar{B}_n^{(j)}(t,q)\frac{z^{n-1}}{(q;q)_{n-1}}$ is
  the right side of~\eqref{restrict}. This finishes the proof of the lemma.
\end{proof}
 \begin{remark}
This interpretation can also be deduced directly from the interpretation in Lemma~\ref{le:4} using Burstein's bijection~\cite{ab}. 
 \end{remark}

For $X\subset[n]$ with $|X|=m$ and $1\in X$, we can define $\bar{\S}_X^{(j)}$  for $1\leq j\leq m$ similarly as $\bar{\S}_m^{(j)}$ like this:
$$
\bar{\S}_X^{(j)}:=\{\pi\in\S_X: \pi(j+1)=1\}\,\, \text{for $1\leq j< m$}\quad\text{and}\quad\bar{\S}_X^{(m)}:=\{\pi' \square 1: \pi'\in\S_{X\setminus\{1\}}\}.
$$
For $1\leq j\leq n$, we define a {\it $j$-restricted two-pix-permutation of $[n]$} to be
 a pair ${\bf v}=(\pi, p_2)$ satisfying:
 \begin{itemize}
 \item $p_2$ (possibly empty) is an increasing word on $[n]$ and 
 \item $\pi\in\bar{\S}_{X}^{(j)}$ with $X=[n]\setminus\cont(p_2)$. 
 \end{itemize}
 Similarly, we define $\lec({\bf v})=\lec(\pi)$ and $\inv({\bf v})=\inv(\pi)+\inv(\cont(\pi),\cont(p_2))$. Let  $\W_{n}^{(j)}$  denote the set of all $j$-restricted two-pix-permutations of $[n]$.

 \begin{lemma}\label{le:7}
 Let $a,j$ be positive integers. Then
\begin{equation} \label{eq:comb}
\sum_{{\bf v}\in\W_n^{(j)}\atop\lec{\bf v}=a}q^{(\inv-\lec){\bf v}}=\sum_{k\geq1}{n-1\brack k-1}_qB_{k,a}^{(j)}(q).
\end{equation}
 \end{lemma}
 \begin{proof} It follows from Lemma~\ref{le:6} and some similar arguments as in the proof of Lemma~\ref{th:3}.
 \end{proof}
 
 \begin{lemma} \label{le:8}Let $2\leq j\leq n$. Then
there is an involution ${\bf v} \mapsto {\bf u}$ 
on $\W_n^{(j)}$ satisfying
\begin{align}\label{pro:symt}
\lec ({\bf v}) = n -2 - \lec ({\bf u}),\quad \textrm{and}\quad
(\inv-\lec) {\bf v} = (\inv-\lec) {\bf u}.
\end{align}
 \end{lemma}
 \begin{proof} Suppose ${\bf v}=(\pi,p_2)\in\W_n^{(j)}$ and $\pi=\tau_0\tau_1\cdots\tau_r$ is the hook factorization of $\pi$ such that $\tau_0$ is a hook or an increasing word and $\tau_i$ ($1\leq i\leq r$) are hooks.
 We also assume that $p_2=x_1\cdots x_l$ if $p_2$ is not empty. Note that $1\notin\cont(\tau_0)$ since $j\neq1$. We will use the involutions $d$ and $d'$ defined in~\eqref{def:d} and~\eqref{def:dd}. There are several cases to be considered:
 \begin{itemize}
 \item[(i)] $\tau_r=\square1$.  Then 
 $$
 {\bf u}=
 \begin{cases}
 (d'(\tau_0)d(\tau_1)\cdots d(\tau_{r-1})x_l1x_1x_2\cdots x_{l-1},\emptyset), &\text{if $p_2\neq\emptyset$};\\
  (d'(\tau_0)d(\tau_1)\cdots d(\tau_{r-1})\square1,\emptyset), &\text{otherwise}.
  \end{cases}
 $$
 \item[(ii)] $\tau_r=y_s1y_1\cdots y_{s-1}$. Then
 $$
 {\bf u}=
 \begin{cases}
 (d'(\tau_0)d(\tau_1)\cdots d(\tau_{r-1})d(\tau_r)d'(p_2),\emptyset), &\text{if $p_2\neq\emptyset$};\\
 (d'(\tau_0)d(\tau_1)\cdots d(\tau_{r-1})\square1,y_1\cdots y_s), &\text{if $p_2=\emptyset$ and  $y_s>y_{s-1}$};\\
  (d'(\tau_0)d(\tau_1)\cdots d(\tau_{r-1})d'(\tau_r),\emptyset), &\text{otherwise}.
  \end{cases}
 $$
 \item[(iii)] $1\notin\cont(\tau_r)$.  Then
  $$
 {\bf u}=
 \begin{cases}
 (d'(\tau_0)d(\tau_1)\cdots d(\tau_{r-1})d(\tau_r)d'(p_2),\emptyset), &\text{if $p_2\neq\emptyset$};\\
 (d'(\tau_0)d(\tau_1)\cdots d(\tau_{r-1}),d'(\tau_r)), &\text{if $p_2=\emptyset$ and  $\lec(\tau_r)=|\tau_r|-1$};\\
  (d'(\tau_0)d(\tau_1)\cdots d(\tau_{r-1})d'(\tau_r),\emptyset), &\text{otherwise}.
  \end{cases}
 $$
 \end{itemize}
 
 First of all, one can check that ${\bf u}\in\W_n^{(j)}$. Secondly, as $d,d'$ are involutions, the above mapping is an involution. Finally, this involution satisfies~\eqref{pro:symt} in all cases. This completes the proof of the lemma.
 \end{proof}
 
 Combining Lemmas~\ref{le:7} and \ref{le:8} we obtain a bijective proof of Theorem~\ref{th:main1}.

\subsection*{Acknowledgement}
The author would like to thank Prof. Jiang Zeng for useful conversations.

\end{document}